\documentclass[a4paper, 11pt]{amsart}
\usepackage[T1]{fontenc}
\usepackage[utf8]{inputenc}
\usepackage{amsmath}
\usepackage{amsfonts}
\usepackage{amssymb}
\usepackage{amsthm}
\usepackage{graphicx}
\usepackage[english]{babel}
\usepackage{version} 
\usepackage{nicefrac}
\usepackage{tipa} 
\usepackage{hyperref}
\usepackage{mathrsfs}
\usepackage{xcolor}
\usepackage{soul}
\usepackage{tikz-cd}

\theoremstyle{definition}
\newtheorem{defin}{Definition}

\theoremstyle{plain}
\newtheorem{theo}[defin]{Theorem}

\renewenvironment{abstract}
{\par\noindent\textbf{\abstractname.}\ \ignorespaces}
{\par\medskip}

\title[The $L^2$-completion of the space of Riemannian metrics is CAT$(0)$]{The $L^2$-completion of the space of Riemannian metrics is CAT$(0)$: a shorter proof}
\author{Nicola Cavallucci}
\address{Nicola Cavallucci, Karlsruhe Institute of Technology, Engelstrasse 2, D-76128 Karlsruhe}
\email{n.cavallucci23@gmail.com}
\date{}

\begin{document}
\maketitle
\begin{abstract}
	\footnotesize
	We reprove in an easier way a result of Brian Clarke: the completion of the space of Riemannian metrics of a compact, orientable smooth manifold with respect to the $L^2$-distance is CAT$(0)$. In particular we show that this completion is isometric to the space of $L^2$-maps from a standard probability space to a fixed CAT$(0)$ space.
\end{abstract}

Ebin (\cite{Eb70}) introduced the so-called $L^2$-metric on the space of all Riemannian metrics on a given compact, orientable manifold. This metric has been extensively studied: its geodesics and curvature has been computed explicitely by Freed and Groisser (\cite{FG89}) and the properties of the exponential map has been studied by Gil-Merano and Michor (\cite{GM91}).
Recently Clarke (\cite{Cla13}) showed that the completion of the $L^2$-metric is a CAT$(0)$ space. In Clarke's proof the thinness of geodesic triangles is verified by direct computations. The scope of this paper is to provide a shorter and more conceptual proof of the same result. As a by-product of our proof we deduce the
\begin{theo}
	\label{theorem-tutti-uguali}
	The completion of the space of Riemannian metrics on a compact, orientable $n$-dimensional, smooth manifold $M$ with respect to the $L^2$-metric depends only on the dimension of $M$.
\end{theo}
In theorem \ref{theo-isometry} we will describe explicitly this completion as the space of $L^2$ maps from $M$ to a fixed CAT$(0)$ metric space: that is why it does not depend on $M$. The CAT$(0)$ property of the target will be detected as a special case of the warped product theorem of \cite{AB04}. It is exactly this identification with a space of $L^2$ maps with target a CAT$(0)$ space that proves the CAT$(0)$ property of the completion. 

\vspace{2mm}

\noindent Let $M$ be a orientable, compact, $n$-dimensional, smooth manifold $M$. We equip $M$ with a smooth Riemannian metric $g_0$ of total volume $1$, with corresponding volume form $\mu_{g_0}$. We denote by $S^2T_p^*M$ the vector space of $(0,2)$-symmetric tensors at $p \in M$ and by $S^2_+T_p^*M$ the subspace of positive-definite ones. 
The disjoint union of the spaces $S^2_+T_p^*M$ defines naturally a fiber bundle $E$ over $M$ with fiber $P(n)$: the space of $n\times n$ real positive definite symmetric matrices. By definition a global smooth section of $E$ is a Riemannian metric on $M$. The set of global smooth sections will be denoted by $\mathcal{E}_{C^\infty}$. On $S^2_+T_p^*M$ we consider the following Riemannian metric. For $h\in S^2_+T_p^*M$ and $a,b \in T_h(S^2_+T_p^*M) \cong S^2T_p^*M$ we set
\begin{equation}
	\label{eq-metric}
	\langle a,b \rangle_{h,p} = \text{tr}(h^{-1}ah^{-1}b)\cdot \sqrt{\det (g_0(p)^{-1} a)}.
\end{equation}
We denote by $d_p$ the metric induced by this Riemannian structure on $S^2_+T_p^*M$. 
Let $\mathcal{E}_m$ be the space of all measurable global sections $\sigma \colon M \to E$, where $M$ and $E$ are equipped with their Borel $\sigma$-algebras. A section $\sigma \in \mathcal{E}_m$ is said to be square integrable if 
$$\int_M d_p^2(\sigma(p), g_0(p))d\mu_{g_0}(p) < \infty.$$
We remark that the integral above is meaningful because the distances $d_p$ vary continuously with $p$, and so the integrand function is measurable. Moreover since $M$ is compact the definition above does not depend on the choice of $g_0$.
We denote the space of square integrable sections by $\mathcal{E}_{\mathcal{L}^2}$. On $\mathcal{E}_{\mathcal{L}^2}$ we define the equivalence relation $\sigma \sim \sigma'$ if and only if $\mu_{g_0}(\lbrace p\in M \text{ s.t. } \sigma(p) \neq \sigma'(p)\rbrace) = 0$. The quotient space is denoted by $\mathcal{E}_{L^2}$. For $\sigma,\sigma' \in \mathcal{E}_{\mathcal{L}^2}$ we set
$$d_{L^2}(\sigma,\sigma') = \left(\int_M d_p^2(\sigma(p), \sigma'(p))d\mu_{g_0}(p)\right)^{\frac12}.$$
An application of the triangle inequality for $d_p$ and the classical inequalities for maps from $M$ to $\mathbb{R}$ shows that $d_{L^2}$ is a finite pseudodistance on $\mathcal{E}_{\mathcal{L}^2}$. Moreover $d_{L^2}(\sigma,\sigma') = 0$ if and only if $\sigma \sim \sigma'$, so $d_{L^2}$ defines a distance on $\mathcal{E}_{L^2}$. Of course every continuous section $\sigma\in \mathcal{E}_{C^0}$ is square integrable. Moreover if we restrict the distance $d_{L^2}$ to the space of smooth sections $\mathcal{E}_{C^\infty}$ we obtain the classical $L^2$-distance (\cite{Cla13}, Theorem 3.8). 
The metric space $(\mathcal{E}_{C^\infty},d_{L^2})$ is not complete because the fibers $(S^2_+T_p^*M, d_p)$ are not complete and because $C^\infty$-sections are not complete in the $L^2$-sense. Its completion has been described in \cite{Cla13comp}, and here we are going to present it in a very natural way.\\
Observe that if $(U,\varphi)$ is a chart of $M$ sending the volume form $\mu_{g_0}$ to the standard volume form of $\mathbb{R}^n$ then the push-forward of the scalar product \eqref{eq-metric} at $p$ on $P(n)$ reads as
\begin{equation}
	\label{eq-metric-P(n)}
	\langle A,B \rangle_{H} = \text{tr}(H^{-1}AH^{-1}B)\cdot \sqrt{\det A},
\end{equation}
for $H\in P(n)$ and $A,B\in T_HP(n)$. The distance metric induced on $P(n)$ by this Riemannian metric is denoted by $d_{P(n)}$. Let $(\bar{P}(n), d_{P(n)})$ be the completion of the metric space $(P(n), d_{P(n)})$. We cover $M$ by charts sending $\mu_{g_0}$ to the Euclidean volume form. The change of coordinate maps between these charts act naturally as isometries of $(P(n), d_{P(n)})$. These isometries can be extended to isometries of $(\bar{P}(n), d_{P(n)})$ and they naturally define a new fiber bundle on $M$ modeled on $\bar{P}(n)$. We denote this fiber bundle by $\bar{E}$. Each fiber $\bar{E}_p$ is naturally equipped with the metric $d_p$ induced by the metric $d_{P(n)}$ on $\bar{P}(n)$. Moreover the fiber bundle $E$ has a natural inclusion map of fiber bundles in $\bar{E}$ and the restriction of this map to each fiber $(E_p,d_p)$ is an isometric embedding in $(\bar{E}_p,d_p)$ with dense image. In other words $(\bar{E}_p,d_p)$ is the completion of $(E_p,d_p)$. We define as above the space $\bar{\mathcal{E}}_{L^2}$, with the corresponding distance $d_{L^2}$. We have the obvious embeddings 
$$\mathcal{E}_{C^\infty} \subset \mathcal{E}_{C^0} \subset \mathcal{E}_{L^2} \subset \bar{\mathcal{E}}_{L^2}.$$
The following result describes the completion of the space of Riemannian metrics with respect to $d_{L^2}$.
\begin{theo}
	\label{theo-inclusions}
	The set $\mathcal{E}_{C^\infty}$ is dense in $\bar{\mathcal{E}}_{L^2}$ with respect to $d_{L^2}$. 
\end{theo}
The proof is standard, we just sketch it leaving the details to the reader.
\begin{proof}
	We fix a triangulation of $M$ and we denote by $V_1,\ldots,V_k$ the interior of the maximal dimensional simplices. We can suppose that each $V_i$ is the support of a trivializing chart for the bundles $\bar{E}$ and $E$ sending $\mu_{g_0}$ to the Euclidean volume form. In order to define a $L^2$-section it is enough to define it on the sets $V_i$ since the complementary of their union has null measure in $M$. A $\bar{E}$-section (resp. $E$-section) on the set $V_i$ can be seen as a map from $V_i$ to $\bar{P}(n)$ (resp. $P(n)$). Let $\sigma \in \bar{\mathcal{E}}_{L^2}$ and $\varepsilon > 0$. The restriction $\sigma_i$ of $\sigma$ to $V_i$ can be seen as a map from $V_i$ to $\bar{P}(n)$. For every $p\in V_i$ we can choose a point $\sigma_i^\varepsilon(p)$ of $P(n)$ such that $d_{P(n)}(\sigma_i^\varepsilon(p), \sigma_i(p)) < \varepsilon$. Gluing together the $\sigma_i^\varepsilon$ we define a section $\sigma^\varepsilon$ of the fiber bundle $E$ such that $\sigma^\varepsilon \in \mathcal{E}_{L^2}$ and $d_{L^2}(\sigma, \sigma^\varepsilon) < \varepsilon \cdot \sqrt{\mu_{g_0}(M)}$. This shows that $\mathcal{E}_{L^2}$ is dense in $\bar{\mathcal{E}}_{L^2}$.\\
	Now we pass to continuous sections. We fix $\sigma \in \mathcal{E}_{L^2}$ and $\varepsilon > 0$.	By Lusin's Theorem we find a subset $W_{i,\sigma}^\varepsilon$ of $V_i$ such that $\mu_{g_0}(V_i\setminus W_{i,\sigma}^\varepsilon) <\varepsilon$ and $\sigma_i$ restricted to $W_{i,\sigma}^\varepsilon$ is continuous. Since $P(n)$ is a convex subset of the vector space $S(n)$ of all $n\times n$ symmetric matrices we can apply Theorem 4.1 of \cite{Dug51} to find a continuous map $\sigma_i^\varepsilon$ defined on $V_i$ such that its restriction to $W_{i,\sigma}^\varepsilon$ coincides with $\sigma_i$ and with image contained in the convex hull of the image of $\sigma_i$, in particular it is contained in $P(n)$.
	In order to glue the pieces together we fix a partition of unity $\lbrace \rho_i \rbrace_{i=1,\ldots,k}$ such that $\rho_i \equiv 1$ on $W_{i,\sigma}^\varepsilon$ and $\rho_i\equiv 0$ outside $V_i$. We set $\sigma^\varepsilon(p) = \sum_{i=1}^k \rho_i(p)\cdot \sigma_i^\varepsilon(p)$. By construction $\sigma^\varepsilon \in \mathcal{E}_{C^0}$ and $d_{L^2}(\sigma^\varepsilon, \sigma) \to 0$ when $\varepsilon$ goes to $0$.\\
	Finally we check that smooth sections are dense in the continuous ones. Use Nash's Theorem to find a Riemannian isometric embedding of $(P(n), d_{P(n)})$ into some Euclidean space. The continuous sections $\sigma_i^\varepsilon$ found above can be approximated by using standard convolution techniques and tubular neighbourhood projections. This, together with partitions of unity as above, allow us to define global smooth sections approximating continuous ones.\\
\end{proof} 
Arguing as at the beginning of the previous proof and using the fact that $L^2$-sections are defined almost everywhere we can simplify the description of the space $\bar{\mathcal{E}}_{L^2}$ as follows.
\begin{theo}
	\label{theo-isometry}
	The metric space $(\bar{\mathcal{E}}_{L^2}, d_{L^2})$ is isometric to $L^2(M, \bar{P}(n))$.
\end{theo}
We recall the definition of $L^2(M, \bar{P}(n))$ and we refer to \cite{KS93}, §1.1 for more details. A function $f\colon M \to \bar{P}(n)$ is said to be square integrable if $\int_M d_{P(n)}^2(f(p), A)d\mu_{g_0}(p) < \infty$ for some (hence any) $A\in \bar{P}(n)$. As usual we define the equivalence relation $f\sim g$ if $f(p)=g(p)$ for $\mu_{g_0}$-a.e.$(p)$ and we denote by $L^2(M, \bar{P}(n))$ the quotient of the set of square integrable functions modulo this equivalence relation. The formula $d_{L^2}(f,g) := \int_M d_{P(n)}^2(f(p), g(p))d\mu_{g_0}(p)$ defines a distance on $L^2(M, \bar{P}(n))$. This is the distance we consider on $L^2(M, \bar{P}(n))$ in the statement of Theorem \ref{theo-isometry}.

\begin{proof}[Proof of Theorem \ref{theo-isometry}]
	We fix a triangulation of $M$ and we denote by $V_1,\ldots,V_k$ the interior of the maximal dimensional simplices. We can suppose that each $V_i$ is the support of a trivializing chart for the bundle $\bar{E}$ sending $\mu_{g_0}$ to the Euclidean volume form. In other words we have the following diagram 
		
	\begin{center}
		\begin{tikzcd} 
			\pi^{-1}(V_i) \arrow[r, "\varphi_i"] \arrow[d, "\pi"']
			& V_i \times \bar{P}(n) \arrow[r, "\pi_2"] & \bar{P}(n) \\ 
			V_i  
		\end{tikzcd}
	\end{center}
	where $\pi$ is the bundle map and $\pi_2$ is the projection on the second factor of the product.	Let $\sigma \in \bar{\mathcal{E}}_{L^2}$ and $i\in \lbrace 1,\ldots,k\rbrace$. On $V_i$ we use the trivialization above to define $f_{\sigma,i}\colon V_i \to \bar{P}(n) = \pi_2\circ \varphi_i \circ \sigma$. Now define $f_\sigma\colon M \to \bar{P}(n)$ as $f_\sigma(p) = f_{\sigma,i}(p)$ if $p\in V_i$ and $f_\sigma(p) = \text{Id} \in P(n)$ if $p\notin \bigcup_i V_i$. The choice of the value of $\Phi$ outside $\bigcup_i V_i$ does not matter because this set has measure $0$. By definitions of the distances on the spaces the map $\Phi \colon \bar{\mathcal{E}}_{L^2} \to L^2(M,\bar{P}(n))$, $\sigma \mapsto f_\sigma$ is an isometry.
\end{proof}
The immediate consequence is the 
\begin{proof}[Proof of Theorem \ref{theorem-tutti-uguali}]
	The space $L^2(M,\bar{P}(n))$ does not depend on $M$. Indeed the measure space $(M,\mathcal{B}, \mu_{g_0})$, where $\mathcal{B}$ is the Borel $\sigma$-algebra of $M$, is a standard probability space, i.e. it is isomorphic as measure space to the interval $[0,1]$ with the standard Lebesgue measure (cp. \cite{dLR93}, Theorem 4.3). It is easy to see that the definition of the space $L^2(M,\bar{P}(n))$ depends only on the measure space class of $M$, so it is isometric to $L^2([0,1],\bar{P}(n))$ and in particular it depends only on $n$.
\end{proof}

Finally we have the shorter proof of Clarke's Theorem, as promised.

\begin{theo}
	The space $(\bar{\mathcal{E}}_{L^2}, d_{L^2})$ is \textup{CAT}$(0)$ and complete. In particular it is the completion of $(\mathcal{E}_{C^\infty}, d_{L^2})$.
\end{theo}
\begin{proof}
	By Theorem \ref{theo-isometry} we know that $\bar{\mathcal{E}}_{L^2}$ is isometric to $L^2(M,\bar{P}(n))$. Therefore it is enough to show that the completion $(\bar{P}(n), d_{P(n)})$ of $(P(n), d_{P(n)})$ is CAT$(0)$ by \cite{KS93}, page 615 or \cite{GN21}, §4.1.1.
	There is another symmetric, complete, nonpositively curved metric $d_0$ on $P(n)$ induced by the Riemannian product
	$$\langle A,B \rangle_H = \text{tr}(H^{-1}AH^{-1}B),$$
	where $H\in P(n)$ and $A,B \in T_HP(n)$. In other words the space $(P(n), d_0)$ is CAT$(0)$. The metric $d_{P(n)}$ is obtained by a conformal change of the metric $d_0$ with conformal factor $H \mapsto \varphi(H) = \sqrt{\det H}$.
	We denote by $P_1(n)$ the subset of $P(n)$ of matrices with determinant $1$: it is a convex subset (cp. \cite{BH09}, Lemma II.10.52). Moreover the map $P_1(n) \times \mathbb{R} \to P(n)$, $(H,t) \mapsto e^{\frac{t}{\sqrt{n}}}H$ is an isometry by \cite{BH09}, Proposition II.10.53. If $H$ is seen as a couple $(H_1,t)$ through this identification then clearly $\varphi$ depends only on $t$: $\varphi(H_1,t) = e^{\frac{t\sqrt{n}}{2}}$. Therefore the metric $d_{P(n)}$ can be seen as the warped product metric $d_0 \times_{e^{\frac{t\sqrt{n}}{2}}} e^{\frac{t\sqrt{n}}{2}}dt$ on $P_1(n) \times \mathbb{R}$.
	The metric space $(-\infty, +\infty)$ with the metric $e^{\frac{t\sqrt{n}}{2}}dt$ is a ray, whose completion is the space $[-\infty, +\infty)$ with the metric $e^{\frac{t\sqrt{n}}{2}}dt$	obtained extending the function $e^{\frac{t\sqrt{n}}{2}}$ by sending $-\infty$ to $0$. Therefore we can construct the new, complete warped product metric space $(P_1(n) \times [-\infty, +\infty), d_0 \times_{e^{\frac{t\sqrt{n}}{2}}} e^{\frac{t\sqrt{n}}{2}}dt)$. The natural isometric embedding of $(P(n) \times \mathbb{R}, d_0 \times_{e^{\frac{t\sqrt{n}}{2}}} e^{\frac{t\sqrt{n}}{2}}dt)$ inside the space above has dense image, so $(P_1(n) \times [-\infty, +\infty), d_0 \times_{e^{\frac{t\sqrt{n}}{2}}} e^{\frac{t\sqrt{n}}{2}}dt)$ is the completion $\bar{P}(n)$ of $(P(n),d_{P(n)})$ and it is CAT$(0)$ because of \cite{AB04}, Theorem 1.1.	
\end{proof}

The explicit description of the completion of the space of Riemannian metrics with respect to the $L^2$-distance given in Theorem \ref{theo-isometry} as $L^2(M,\bar{P}(n))$ tells us that its isometry group is larger than expected. Indeed not only it contains the group of volume-preserving diffeomorphisms of any compact, orientable, $n$-dimensional smooth manifold, but also the much larger group of all automorphisms of the standard probability space $[0,1]$ endowed with the Lebesgue measure. It could be interesting to determine explicitly the whole isometry group of this space. An additional question could be to find the description of the boundary at infinity of $L^2(M,\bar{P}(n))$.

\vspace{10mm}

\small
\noindent {\sc Acknowledgments.} {\em I thank A. Lytchak for introducing me to the problem studied in this paper and for teaching me the general philosophy behind this kind of questions.}
\normalsize

\bibliographystyle{alpha}
\bibliography{space_riemannian_metrics}

\end{document}